\documentclass[12pt,twoside]{article}

\usepackage{fancyhdr}
\usepackage{amsmath}
\usepackage{amscd}
\usepackage{amssymb}
\usepackage{amsfonts}
\usepackage{euscript}
\usepackage{oldgerm}
\usepackage{calligra}
\usepackage{mathrsfs}
\usepackage{hyperref}
\usepackage{url}
\usepackage{lastpage}
\usepackage{multirow}
\usepackage{graphicx}
\usepackage{stmaryrd}
\usepackage{wasysym}
\usepackage{pifont}

\usepackage[T1]{fontenc}
\usepackage[latin1]{inputenc}

\usepackage{draftwatermark}

\SetWatermarkAngle{45}
\SetWatermarkLightness{0.9}
\SetWatermarkFontSize{5cm}
\SetWatermarkScale{2}
\SetWatermarkText{ }

\renewcommand{\thefootnote}{\fnsymbol{footnote}}

\long\def\sfootnote[#1]#2{\begingroup
\def\thefootnote{\fnsymbol{footnote}}\footnote[#1]{#2}\endgroup}

\newtheorem{theorem}{Theorem}[section]
\newtheorem{definition}[theorem]{Definition}

\newtheorem{remark}[theorem]{Remark}

\newenvironment{proof}{\noindent\mbox{\bf Proof.}}
{\hfill\mbox{\ding{111}}\bigskip}

\begin{document}

\pagestyle{fancy}
\lhead[]{}
\chead[{\bf  Theorems of Tarski's Undefinability  and G\"odel's $2^{\rm nd}$ Incompleteness---Computationally}]{{\bf  Theorems of Tarski's Undefinability  and G\"odel's $2^{\rm nd}$ Incompleteness---Computationally}}
\rhead[]{}
\lfoot[{\tt page \thepage \ (of \pageref{LastPage})}]{}
\cfoot[]{}
\rfoot[]{{\tt page \thepage \ (of \pageref{LastPage})}}
\renewcommand{\headrulewidth}{1pt}
\renewcommand{\footrulewidth}{1pt}
\thispagestyle{empty}

\begin{table}
\begin{center}
\hspace{0.75em}
\begin{tabular}{|| c || l  | l ||}

\hline
 \multirow{6}{*}{\includegraphics[scale=0.175]{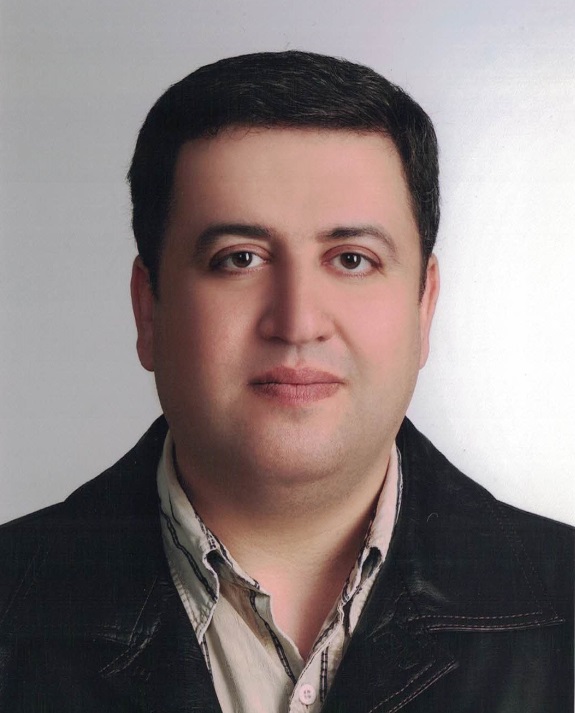}}&    &
 \multirow{7}{*}{ \ \ } \ \ \ \ \\
 &     \ \ {\large{\sc Saeed Salehi}}  \ \  \ & \ \    Tel: \, +98 (0)41 3339 3930      \\
  &   \ \ University of Tabriz \ \ \  & \ \ E-mail: \!\!{\tt /root}{\sf @}{\tt SaeedSalehi.ir/}     \\
 &  \ \ P.O.Box 51666--16471 \ \ \ &   \ \ \ \ {\tt /SalehiPour}{\sf @}{\tt TabrizU.ac.ir/}     \\
 &   \ \ Tabriz, IRAN \ \ \ & \ \ Web: \  \ {\tt http:\!/\!/SaeedSalehi.ir/}    \\
 &    &    \\
 \hline
\end{tabular}
\end{center}
\bigskip
\end{table}

\vspace{1.5em}

\begin{center}

\bigskip

{\bf {\Large  Theorems of Tarski's Undefinability   and \\[1ex] G\"odel's Second Incompleteness---Computationally
}}
\end{center}

\vspace{1.5em}

\begin{abstract}\noindent
  We present a version of G\"odel's Second Incompleteness Theorem for recursively enumerable consistent extensions of a fixed  axiomatizable theory, by incorporating some bi-theoretic version of the derivability conditions. We  also argue that Tarski's theorem on the Undefinability of Truth is G\"odel's  First Incompleteness Theorem relativized to definable oracles;  a unification of these two theorems is given.

\bigskip

\centerline{${\backsim\!\backsim\!\backsim\!\backsim\!\backsim\!\backsim\!\backsim\!
\backsim\!\backsim\!\backsim\!\backsim\!\backsim\!\backsim\!\backsim\!
\backsim\!\backsim\!\backsim\!\backsim\!\backsim\!\backsim\!\backsim\!
\backsim\!\backsim\!\backsim\!\backsim\!\backsim\!\backsim\!\backsim\!
\backsim\!\backsim\!\backsim\!\backsim\!\backsim\!\backsim\!\backsim\!
\backsim\!\backsim\!\backsim\!\backsim\!\backsim\!\backsim\!\backsim\!
\backsim\!\backsim\!\backsim\!\backsim\!\backsim\!\backsim\!\backsim\!
\backsim\!\backsim\!\backsim\!\backsim\!\backsim\!\backsim\!\backsim\!
\backsim\!\backsim\!\backsim\!\backsim\!\backsim\!\backsim\!\backsim\!
\backsim\!\backsim\!\backsim\!\backsim\!\backsim\!\backsim\!\backsim\!
\backsim\!\backsim\!\backsim}$}

\bigskip



\bigskip

\bigskip

\bigskip

%{\footnotesize
%\noindent {\bf Acknowledgements} \
%The author is partially supported by grant  $\textsf{N}^{\underline{\sf o}}$~93030033 from $\bigcirc\hspace{-2ex}{\not}\hspace{.25ex}
%\bullet\hspace{-1.65ex}{\not}\bigcirc$ $\mathbb{I}\mathbb{P}\mathbb{M}$.
%}

\end{abstract}

\bigskip
\bigskip
\bigskip
\bigskip

\hspace{.75em} \fbox{\textsl{\footnotesize Date: 10 November 2019  (10.11.19)}}

\vfill

\bigskip
\noindent\underline{\centerline{}}
\centerline{\tt page 1 (of \pageref{LastPage})}

%%%
%%% the paper begins ...
%%%

\newpage
\setcounter{page}{2}
\SetWatermarkAngle{65}
\SetWatermarkLightness{0.925}
%\SetWatermarkFontSize{30cm}
\SetWatermarkScale{2.45}
\SetWatermarkText{\hspace{-1.25em}{\tt \copyright}\ {\sf 2019} {\sc Saeed Salehi}}

%%%%%%%%%%%%%%%%%%%%%%%%%%%%%%%%%%%%
%%%%%%%%%%%%%%%%%%%%%%%%%%%%%%%%%%%%
%%%%%%%%%%%%%%%%%%%%%%%%%%%%%%%%%%%%
%%%%%%%%%%%%%%%%%%%%%%%%%%%%%%%%%%%%
%%%%%%%%%%%%%%%%%%%%%%%%%%%%%%%%%%%%

%%%%%%%%%%%%%%%%%%%%%%%%%%%%%%%%%%%%
%%%%%%%%%%%%%%%%%%%%%%%%%%%%%%%%%%%%
%%%%%%%%%%%%%%%%%%%%%%%%%%%%%%%%%%%%
%%%%%%%%%%%%%%%%%%%%%%%%%%%%%%%%%%%%
%%%%%%%%%%%%%%%%%%%%%%%%%%%%%%%%%%%%
%--------------------------------------

\section{Introduction}
In this paper we will argue that Tarski's theorem on the undefinability of {\sc Truth} in sufficiently expressive languages, which on its face value has nothing to do with (oracle) computations, is equivalent with G\"odel's (semantic form of the) first incompleteness theorem relativized to definable oracles.  Actually, we will show a theorem which unifies the theorems of G\"odel and Tarski. Then we will discuss G\"odel's Second Incompleteness Theorem. Since this  theorem,  inability of  sufficiently strong theories to prove (a statement of) their  own consistency, is not robust with respect to the notion of consistency, its proof is much more delicate and elegant than the proof of the first theorem; indeed the proof appears in very few places (see \cite{salehi-review}  a review  of the first edition of \cite{smith}). Though, some {\sc book} proofs (in the words of Paul Erd\H{o}s) for the first incompleteness theorem exist in the literature, a nice and neat proof (understandable to the undergraduates or amateur mathematicians) for the second theorem is missing. We will present a proof for this theorem from computational viewpoint which will be based on some bitheoretic derivability conditions.

\section{Unifying Theorems of G\"odel and Tarski}
Here,  we establish a relation  between G\"odel's First Incompleteness Theorem (in its weaker semantic form) and Tarski's Theorem on the Undefinability of Truth; indeed, we prove a theorem which unifies these two.
\subsection{Finitely Given Infinite Sets (of Natural Numbers)}
How can an infinite set (such as $\{0,3,6,9,\cdots,3k,\cdots\}$  or $\{0,1,4,9,\cdots,k^2,\cdots\}$) be finitely given? (We consider sets of natural numbers, i.e., subsets of $\mathbb{N}$, throughout the paper). There are a few definitions for this concept in the literature such as:

\begin{definition}{\rm

\noindent

\noindent

\begin{itemize}
\item A set $D(\subseteq\mathbb{N})$ is {\em decidable}, when there exists a
  single-input and Boolean-output algorithm which on any input $x(\in\mathbb{N})$ outputs {\tt Yes}, if $x\in D$, and   outputs {\tt No}, if $x\not\in D$.
\item A set $R(\subseteq\mathbb{N})$ is called {\em recursively enumerable}
  ({\sc re} for short), when there exists  an input-free algorithm which outputs
  (generates) the elements of $R$ (after running).
\item A set $S(\subseteq\mathbb{N})$ is {\em semi-decidable}, when there exists a single-input and output-free algorithm which after running on an input $x\in\mathbb{N}$ halts if and only if $x\in S$ (and so when $x\not\in S$ the algorithms runs forever on input $x$). \hfill\ding{71}
\end{itemize}
}\end{definition}
Two deep theorems of Computability Theory (see e.g. \cite{epstein}) state that

\bigskip

\qquad \qquad $\lozenge\!\!\!\blacklozenge$ {\em semi-decidability is equivalent to being an {\sc re} set}, and

\bigskip

\qquad \qquad  $\lozenge\!\!\!\blacklozenge$ {\em decidability is equivalent to recursively enumerability of a set and its complement}.

\bigskip

Fix a (sufficiently expressive) language of arithmetic, like the language $\{0,S,+,\times,\leqslant\}$ (as in  \cite{hajekpudlak}) or $\{0,1,+,\times,<\}$ (as in  \cite{kaye}), and denote by $\mathcal{N}$ the structure of $\mathbb{N}$ by this language, so that $\mathcal{N}\models\varphi$ makes the sense that the sentence $\varphi$ (in this language of arithmetic) holds true in the set of natural numbers (equipped with this language of arithmetic by their standard interpretations).

\bigskip

\begin{definition}{\rm

\noindent

\noindent

\begin{itemize}
\item A set $A(\subseteq\mathbb{N})$ is {\em definable}, when there exists a formula $\varphi(x)$, in the language of arithmetic,  such that $A=\{n\in\mathbb{N}\mid\mathcal{N}\models\varphi(\overline{n})\}$, where the term $\overline{n}$ represents $n$ in the language of arithmetic (which could be $S\cdots S(0)$ or $1+\cdots+1$ [for $n$-times]). \hfill\ding{71}
\end{itemize}
}\end{definition}

\begin{definition}{\rm
The classes of formulas $\{\Sigma_n\}_{n\in\mathbb{N}}$ and $\{\Pi_n\}_{n\in\mathbb{N}}$ are defined in the standard way \cite{hajekpudlak,kaye}: $\Sigma_0=\Pi_0$ is the class of bounded formulas (in which every universal quantifier has the form $\forall x ([x\leqslant t \rightarrow\cdots]$ and every existential quantifier has the form $\exists x [x\leqslant t\wedge\cdots]$ for some term $t$), and the class $\Sigma_{n+1}$ contains the closure of $\Pi_n$ under the existential quantifiers, and is closed under disjunction, conjunction, existential quantifiers and bounded universal quantifiers; similarly,  the class $\Pi_{n+1}$ contains the closure of $\Sigma_n$ under the universal quantifiers, and  is closed under disjunction, conjunction, universal quantifiers and bounded existential quantifiers. Let us also define  $\Delta_n=\Sigma_n\cap\Pi_n$ and  note that the negation of a $\Sigma_n$-formula is a $\Pi_n$-formula, and vice versa.
}\hfill\ding{71}\end{definition}
Another deep fact from Computability Theory (and Mathematical Logic) is that

\bigskip

\qquad \qquad $\lozenge\!\!\!\blacklozenge$  {\em {\sc re} sets are exactly the sets definable by $\Sigma_1$-formulas}.

\bigskip

And so, co-{\sc re} sets are the ones definable by $\Pi_1$-formulas, and then the decidable sets coincide with $\Delta_1$-definable sets. So, one can say that in a sense
$$\textsf{Computability is Definability.}$$
And conversely, {\sf definability is (relativized) computability (by oracles)}: since having an oracle for deciding the arithmetical formula $\varphi(x)$ we can decide whether given $n$ belongs to the set defined by $\varphi(x)$ or not (if $\varphi(\overline{n})$ holds then it does belong to the set and if $\varphi(\overline{n})$ does not hold then $n$ does not belong to the set).

Thus, by a finitely given (infinite) set we may mean  a definable set, the complexity of whose definition describes the complexity of the computation of the membership algorithm in that set.
\subsection{Computability vs. Definability}
G\"odel's First Incompleteness theorem in its weaker semantic form states that the set of all true arithmetical sentences ${\rm Th}(\mathcal{N})$ (see the definition below)  is not decidable. It immediately follows that this set is not {\sc re}, since otherwise its complement, noting that  ${\rm Th}(\mathcal{N})^\complement=\{\neg\theta\mid\theta\!\in\!{\rm Th}(\mathcal{N})\}$, would be {\sc re} too, and thus ${\rm Th}(\mathcal{N})$ would be {\sc re} and co-{\sc re} and whence decidable!
Hence, in the semantic form G\"odel's first incompleteness theorems states that ${\rm Th}(\mathcal{N})\not\in\Sigma_1$ (by abusing the notation we may denote by $\Sigma_n$ the class of $\Sigma_n$-formulas and also the class of definable sets by some $\Sigma_n$-formulas). Syntactically, this theorem of G\"odel is usually stated as ``{\em no sound and {\sc re} extension of Peano's Arithmetic can be complete}''; in notation  (where ${\tt PA}$ denotes Peano's Arithmetic)

 ${\tt PA}\subseteq T\;\&\;T\!\in\!\Sigma_1\;\&\;T\subseteq{\rm Th}(\mathcal{N})\Longrightarrow T\neq{\rm Th}(\mathcal{N})$.
\begin{definition}{\rm
The  set of all true arithmetical sentences is

${\rm Th}(\mathcal{N})=\{\theta\in\textrm{Arithmeical Sentences}\mid\mathcal{N}\models\theta\}$.

\noindent For any $n$, the  set of all true arithmetical $\Sigma_n$-sentences is

$\Sigma_n\text{-}{\rm Th}(\mathcal{N})=\{\theta\!\in\!\Sigma_n\text{-}{\rm Sentences}\mid\mathcal{N}\models\theta\}$.

\noindent  Also, the  set of all true arithmetical $\Pi_n$-sentences is

$\Pi_n\text{-}{\rm Th}(\mathcal{N})=\{\theta\!\in\!\Pi_n\text{-}{\rm Sentences}\mid\mathcal{N}\models\theta\}$.
}\hfill\ding{71}\end{definition}
\noindent Still, a more precise reading of G\"odel's first incompleteness theorem is:

\bigskip

$(\star)\qquad {\tt PA}\subseteq T \; \& \; T\in\Sigma_1 \; \& \; T\subseteq{\rm Th}(\mathcal{N}) \Longrightarrow \Pi_1\text{-}{\rm Th}(\mathcal{N})\not\subseteq T$

\bigskip

\noindent since G\"odel's true but unprovable (in $T$) sentence is indeed $\Pi_1$.

Tarski's theorem on the undefinability of truth states that ${\rm Th}(\mathcal{N})$ is not definable. For this theorem to make sense we should view ${\rm Th}(\mathcal{N})$ as a set of natural numbers, and this is done by a fixed (standard) G\"odel numbering of syntax. Let $\ulcorner\alpha\urcorner$ denote the G\"odel number of the object $\alpha$. So, (again by abusing the notation) identifying the set of natural numbers $$\{\ulcorner\theta\urcorner\mid \theta\in\textrm{Arithmetical Sentences}\;\&\; \mathcal{N}\models\theta\}$$ with ${\rm Th}(\mathcal{N})$, we can talk about definability or undefinability of ${\rm Th}(\mathcal{N})$. So, Tarski's theorem states that for any $n$, ${\rm Th}(\mathcal{N})\not\in\Sigma_n$. For the sake of unifying it with G\"odel's theorem let us present this theorem as

\bigskip

$(\ast)_n\qquad {\tt PA}\subseteq T \; \& \; T\in\Sigma_n \; \& \; T\subseteq{\rm Th}(\mathcal{N}) \Longrightarrow {\rm Th}(\mathcal{N})\not\subseteq T$

\bigskip

\noindent stating that \  ``no definable and sound extension of  ${\tt PA}$ can be complete''. Compare  with   G\"odel's theorem  stated above:\! ``no $\Sigma_1$-definable and sound extension of  ${\tt PA}$ can be $\Pi_1$-complete''.

\begin{theorem}[\cite{salehi-seraji}]\label{thm-first}
No  $\Sigma_n$-definable and sound   extension of  ${\tt PA}$ is $\Pi_n$-complete  (for any $n\!>\!0$).

$(\divideontimes)_n\qquad {\tt PA}\subseteq T \; \& \; T\in\Sigma_n \; \& \; T\subseteq{\rm Th}(\mathcal{N}) \Longrightarrow \Pi_n\text{-}{\rm Th}(\mathcal{N})\not\subseteq T$
\end{theorem}
\begin{proof}
If $T$ is $\Sigma_n$-definable, then so is its provability predicate ${\sf Pr}_T(x)$. By (G\"odel-Carnap's) Diagonal Lemma there exists an arithmetical sentence $\boldsymbol\gamma$ such that the equivalence ${\tt PA}\vdash \boldsymbol\gamma\longleftrightarrow\neg{\sf Pr}_T(\ulcorner\boldsymbol\gamma\urcorner)$ holds. This sentence $\boldsymbol\gamma$ is equivalently a $\Pi_n$-sentence (and even can be explicitly constructed to be so). Now, we show that $\mathcal{N}\models\boldsymbol\gamma$. Since, otherwise (if $\mathcal{N}\models\neg\boldsymbol\gamma$, then) $\mathcal{N}\models{\sf Pr}_T(\ulcorner\boldsymbol\gamma\urcorner)$ and so $T\vdash\boldsymbol\gamma$, but this contradicts the soundness of $T$. So, $\boldsymbol\gamma\in\Pi_n\text{-}{\rm Th}(\mathcal{N})$. Finally, we show that $T\not\vdash\boldsymbol\gamma$. Because, if $T\vdash\boldsymbol\gamma$, then (by ${\tt PA}\subseteq T$) $T\vdash\neg{\sf Pr}_T(\ulcorner\boldsymbol\gamma\urcorner)$ and so (by soundness of $T$) $\mathcal{N}\models\neg{\sf Pr}_T(\ulcorner\boldsymbol\gamma\urcorner)$; whence $T\not\vdash\boldsymbol\gamma$, a contradiction.
\end{proof}

\begin{remark}{\rm
Obviously, $(\divideontimes)_1$ is the same as $(\star)$, and also  $(\divideontimes)_n$  implies $(\ast)_n$ for every $n>0$. Thus, Theorem~\ref{thm-first} implies G\"odel's First Incompleteness Theorem (for~$n=1$) and also Tarski's Theorem on the Undefinability of Truth.
%(noting that $\Pi_n\text{-}{\rm Th}(\mathcal{N})\subseteq{\rm Th}(\mathcal{N})$).
}\hfill\ding{71}\end{remark}

\section{G\"odel's Second Incompleteness Theorem}
G\"odel's Second Incompleteness Theorem states that sufficiently strong theories (in sufficiently expressive languages)  cannot prove their own consistency. To make it more precise, it should read as: for a sufficiently strong theory in a sufficiently expressive language there exists a sentence which expresses the consistency of the theory (in a way or another) which is not provable from the theory.

A classical proof of this theorem (which is not its only proof) goes roughly as:

\noindent \underline{\; 1 \;} First, the consistency statement ${\sf Con}_T$ of a theory $T$ comes from a provability predicate ${\sf Pr}_T$ of that theory by the definition $\neg{\sf Pr}_T(\ulcorner\bot\urcorner)$, where $\bot$ is a  contradictory statement such as $t\neq t$ for a term $t$ in the language of $T$.

\noindent \underline{\; 2 \;} Second, this provability predicate should satisfy some conditions, the most famous of which are  the following which are known as Hilbert-Bernays-L\"ob provability (or derivability) conditions: for any  $\varphi,\psi$,

\begin{tabular}{ll}
(i) & $T\vdash{\sf Pr}_T(\ulcorner\varphi\urcorner)$, if $T\vdash\varphi$; \\
(ii) & $T\vdash{\sf Pr}_T(\ulcorner\varphi\rightarrow\psi\urcorner)\rightarrow\big[{\sf Pr}_T(\ulcorner\varphi\urcorner)\rightarrow{\sf Pr}_T(\ulcorner\psi\urcorner)\big]$; \\
(iii) & $T\vdash{\sf Pr}_T(\ulcorner\varphi\urcorner)\rightarrow {\sf Pr}_T(\ulcorner{\sf Pr}_T(\ulcorner\varphi\urcorner)\urcorner)$.
\end{tabular}

\noindent
These conditions translate nicely to the language of modal logic when $\square$ is interpreted as provability:

 (i) is the same as the necessitation rule $\varphi/\square\varphi$,

 (ii) is the same as the Kripke's distribution axiom $\square(\varphi\rightarrow\psi)\rightarrow(\square\varphi\rightarrow\square\psi)$,

 (iii) is the same as what is called the $4$ axiom in modal logic $\square\varphi\rightarrow\square\square\varphi$.

\noindent \underline{\; 3 \;} Third, finally, the classical proof uses  Diagonal Lemma, just like the proof of G\"odel's First Incompleteness Theorem, for the formula ${\sf Pr}_T(\ulcorner\xi\urcorner)\rightarrow\bot$ to get a formula $\mathscr{G}$  which satisfies the following provable equivalence

$(\textswab{d}) \qquad T\vdash \mathscr{G} \longleftrightarrow \big[{\sf Pr}_T(\ulcorner \mathscr{G}\urcorner)\rightarrow\bot\big]$

\noindent \underline{\; \; \;} Then, for the sake of a contradiction, assuming that

$(0) \qquad T\vdash {\sf Con}(T)$

\noindent the proof continues as follows (note that for inferring the items (1) and (5) we use the tautologies $(\neg A) \equiv (A\rightarrow\bot)$ and   $A\rightarrow (B\rightarrow C) \equiv (A\rightarrow B)\rightarrow(A\rightarrow C)$, respectively):

\begin{tabular}{lll}
(1) & $T\vdash {\sf Pr}_T(\ulcorner\bot\urcorner)\rightarrow\bot$ & by $(0)$\\
(2) & $T\vdash \mathscr{G}\rightarrow \big[{\sf Pr}_T(\ulcorner\mathscr{G}\urcorner)\rightarrow\bot\big]$ & by $(\textswab{d})$ \\
(3) & $T\vdash {\sf Pr}_T\big(\ulcorner\mathscr{G}\rightarrow \big[{\sf Pr}_T(\ulcorner\mathscr{G}\urcorner)\rightarrow\bot\big]\urcorner\big)$  & by (2),(i)\\
(4) & $T\vdash {\sf Pr}_T(\ulcorner\mathscr{G}\urcorner) \rightarrow \big[{\sf Pr}_T\big(\ulcorner{\sf Pr}_T(\ulcorner\mathscr{G}\urcorner)\urcorner\big)\rightarrow {\sf Pr}_T(\ulcorner\bot\urcorner)\big]$ & by (3),(ii)\\
(5) & $T\vdash {\sf Pr}_T(\ulcorner\mathscr{G}\urcorner) \rightarrow  {\sf Pr}_T(\ulcorner\bot\urcorner)$ & by (4),(iii)\\
(6) & $T\vdash {\sf Pr}_T(\ulcorner\mathscr{G}\urcorner) \rightarrow  \bot$ & by (5),(1)\\
(7) & $T\vdash \mathscr{G}$ & by (6),$(\textswab{d})$\\
(8) & $T\vdash {\sf Pr}_T(\ulcorner\mathscr{G}\urcorner)$ & by (7),(i)\\
(9) & $T\vdash\bot$ & by (6),(8)
\end{tabular}

\noindent So, if $T$ is consistent (and satisfies the provability conditions), then $T\not\vdash{\sf Con}(T)$.

As a matter of fact, the above proof proves much more than G\"odel's second incompleteness theorem. If $\bot$ is replaced with $\varphi$ in (1)--(10), then L\"ob's rule is derived: if $T\vdash{\tt Pr}_T(\ulcorner\varphi\urcorner)\rightarrow\varphi$,  then $T\vdash\varphi$.
%This implies that $T\vdash H$ for any formula $H$ which satisfies $T\vdash H\leftrightarrow {\sf Pr}_T(\ulcorner H\urcorner)$, answering a question of Henkin.
Almost
the same line  of reasoning  can show L\"ob's Axiom: $$T\vdash {\sf Pr}_T\big(\ulcorner{\sf Pr}_T(\ulcorner\varphi\urcorner)\rightarrow\varphi\urcorner\big)\rightarrow{\sf Pr}_T(\ulcorner\varphi\urcorner)$$ which immediately (by contraposition) implies the formalized form of G\"odel's Second Incompleteness Theorem: $$T\vdash {\sf Con}(T\!\cup\!\{\neg\varphi\})\rightarrow\neg{\sf Pr}_{T\cup\{\neg\varphi\}}\big(\ulcorner{\sf Con}(T\!\cup\!\{\neg\varphi\})\urcorner\big).$$ In particular for $\varphi=\bot$ we get $T\vdash {\sf Con}(T)\rightarrow\neg{\sf Pr}_{T}\big(\ulcorner{\sf Con}(T)\urcorner\big)$, which is exactly what G\"odel's second incompleteness theorem states: if a theory is consistent  (and satisfies some conditions), then it cannot prove its own consistency. To emphasize the importance of this generalization and showing the strength of this classical proof we present a proof for L\"ob's axiom below: for a given sentence $\varphi$, by Diagonal Lemma, there exists a sentence $\mathscr{G}$ such that

\quad $(\textswab{d}) \quad T\vdash \mathscr{G} \longleftrightarrow \big[{\sf Pr}_T(\ulcorner \mathscr{G}\urcorner)\rightarrow\varphi\big]$.

\noindent Now, we reason for the theory $\widehat{T}=T+{\sf Pr}_T\big(\ulcorner{\sf Pr}_T(\ulcorner\varphi\urcorner)\rightarrow\varphi\urcorner\big)$ as follows

\begin{tabular}{lll}
(1) & $T\vdash {\sf Pr}_T\big(\ulcorner\mathscr{G}\rightarrow \big[{\sf Pr}_T(\ulcorner\mathscr{G}\urcorner)\rightarrow\varphi\big]\urcorner\big)$ & by $(\textswab{d})$,(i) \\
(2) & $T\vdash {\sf Pr}_T(\ulcorner\mathscr{G}\urcorner) \rightarrow \big[{\sf Pr}_T\big(\ulcorner{\sf Pr}_T(\ulcorner\mathscr{G}\urcorner)\urcorner\big)\rightarrow {\sf Pr}_T(\ulcorner\varphi\urcorner)\big]$ & by (1),(ii) \\
(3) & $T\vdash {\sf Pr}_T(\ulcorner\mathscr{G}\urcorner) \rightarrow  {\sf Pr}_T(\ulcorner\varphi\urcorner)$ & by (3),(iii) \\
(4) & $T\vdash {\sf Pr}_T\big(\ulcorner{\sf Pr}_T(\ulcorner\mathscr{G}\urcorner) \rightarrow  {\sf Pr}_T(\ulcorner\varphi\urcorner)\urcorner\big)$ & by (3),(i) \\
(5) & $\widehat{T}\vdash {\sf Pr}_T\big(\ulcorner{\sf Pr}_T(\ulcorner\mathscr{G}\urcorner) \rightarrow  \varphi\urcorner\big)$ & by (4)$\&$hyp. \\
(6) & $T\vdash{\sf Pr}_T\big(\ulcorner  \big[{\sf Pr}_T(\ulcorner\mathscr{G}\urcorner)\rightarrow\varphi\big] \rightarrow \mathscr{G}\urcorner\big)$ & by $(\textswab{d})$,(i)\\
(7) & $T\vdash{\sf Pr}_T\big(\ulcorner  {\sf Pr}_T(\ulcorner \mathscr{G}\urcorner)\rightarrow\varphi\urcorner\big) \rightarrow {\sf Pr}_T(\ulcorner \mathscr{G}\urcorner)$ & by (6),(ii)\\
(8) & $\widehat{T}\vdash {\sf Pr}_T(\ulcorner \mathscr{G}\urcorner)$ & by (5),(7) \\
(9) & $\widehat{T}\vdash {\sf Pr}_T(\ulcorner\varphi\urcorner)$ & by (3),(8)
\end{tabular}

\noindent Let us note that  (5) follows from (4) and the hypothesis ${\sf Pr}_T\big(\ulcorner{\sf Pr}_T(\ulcorner\varphi\urcorner)\rightarrow\varphi\urcorner\big)$ with the following formula which holds by (i) and (ii)
$${\sf Pr}_T(\ulcorner\mathcal{A} \rightarrow  \mathcal{B}\urcorner)\longrightarrow\big[{\sf Pr}_T(\ulcorner\mathcal{B}\rightarrow\mathcal{C}\urcorner)\longrightarrow{\sf Pr}_T(\ulcorner\mathcal{A} \rightarrow \mathcal{C}\urcorner)\big],$$
by putting $\mathcal{A}={\sf Pr}_T(\ulcorner G\urcorner)$, $\mathcal{B}={\sf Pr}_T(\ulcorner\varphi\urcorner)$, $\mathcal{C}=\varphi$.

So, this argument which has become classical in modern textbooks (see \cite{smorynski-modal} for a historical account of Hilbert-Bernays-L\"ob provability conditions) is too strong; it can indeed prove a formalized version of G\"odel's second theorem and even more. Another dilemma with this proof is that it appears in very few places, since most of the authors know the proof through the provability conditions (i), (ii) and (iii). Though (i) and (ii) can be proved rather easily, the proof of (iii) is rather rare   (see~\cite{salehi-review}). Indeed, for (i) one needs to know/show that
\begin{itemize}
\item for any {\sc re} theory $T$, the formula ${\sf Pr}_T$ is $\Sigma_1$;
\item most natural theories, like ${\tt PA}$, are $\Sigma_1$-complete;
\item so if $T$ is an {\sc re} theory containing ${\tt PA}$, then (i) holds for ${\sf Pr}_T$.
\end{itemize}
For (ii) it suffices to note that if the formula ${\sf Proof}_T(x,y)$ represents the statement ``$y$ is the G\"odel number of a proof in $T$ of the formula with G\"odel number $x$'' (so ${\sf Pr}_T(x)\equiv\exists y\,{\sf Proof}_T(x,y)$ by definition), then (ii) is equivalent to  $${\sf Proof}_T(\ulcorner\varphi\rightarrow\psi\urcorner,u)\longrightarrow\big[{\sf Proof}_T(\ulcorner\varphi\urcorner,v)\longrightarrow\exists w{\sf Proof}_T(\ulcorner\psi\urcorner,w)\big].$$
Having $u$ and $v$, it suffices to take $w$ as $u^\frown v^\frown\ulcorner\psi\urcorner$, where $^\frown$ denotes concatenation (of strings).  So, if
\begin{itemize}
\item $T$ can prove the totality of  concatenation, i.e., $T\vdash   \forall u,v\exists w (u^\frown v=w)$,
\end{itemize}
    then (ii) holds for $T$.

\noindent But, as mentioned before, the proof of (iii) is rather technical  and so appears in many few places. One reason is that (iii) cannot be (easily) proved directly; indeed, its proof goes through proving the formalized $\Sigma_1$-completeness for the theory $T$:

\begin{tabular}{ll}
(iv) & $T\vdash\sigma\rightarrow{\sf Pr}_T(\ulcorner\sigma\urcorner)$ for any $\Sigma_1$-formula $\sigma$.
\end{tabular}

\noindent It is interesting to note that the third provability condition sometimes is taken to be (iv), rather than (iii) which is a special case of (iv). All the existing proofs of (iii) indeed prove (iv).
It is actually difficult to prove $T\vdash\sigma\rightarrow{\sf Pr}_T(\ulcorner\sigma\urcorner)$, for any $\Sigma_1$-formula $\sigma$,  for particular $T$'s like ${\tt PA}$.  Let us note that (i),(ii), (iii) and (iv) involve a kind of self-reference: the theory $T$ can prove some statements about its own provability predicate. The fundamental question here is that: \textsl{does every proof of G\"odel's Second Incompleteness Theorem have to go through proving} (iii) {\sl or} (iv)?
 Fortunately, the answer is no! and some beautiful proofs of this theorem can be found in e.g. \cite{adamow,jech,kotlarski94,kotlarski98} some of which even avoid the use of  Diagonal Lemma (cf. also \cite{salehi2014} for a diagonal-free proof of G\"odel-Rosser's theorem).

\subsection{Bi-Theoretic Derivability Conditions}
``For certain purposes'' in particular  ``for what is perhaps the most important philosophical application of'' G\"odel's Second Incompleteness Theorem ``namely, that to Hilbert's Program'', Detlefsen \cite{detlefsen} introduced a bitheoretic version of this theorem; a version ``in which it is allowed that the representing and represented theories be different.''
Our remark above about the circularity of the derivability conditions (i,ii,iii,iv) was based on the fact that a single theory does all the job: prove some facts about its own provability. For example in (i) we have that $T\vdash{\sf Pr}_T(\ulcorner\varphi\urcorner)$, whenever $T\vdash\varphi$. But the fact of the matter is that if $T\vdash\varphi$, then we also have that
${\tt PA}\vdash{\sf Pr}_T(\ulcorner\varphi\urcorner)$ (and so when ${\tt PA}\subseteq T$ we can conclude that $T\vdash{\sf Pr}_T(\ulcorner\varphi\urcorner)$). It seems to us that the new bitheoretic condition

$T\vdash\varphi \Longrightarrow {\tt PA}\vdash{\sf Pr}_T(\ulcorner\varphi\urcorner)$

\noindent is somehow stronger than the monotheoretic condition

$T\vdash\varphi \Longrightarrow T\vdash{\sf Pr}_T(\ulcorner\varphi\urcorner)$

\noindent even if we assume that ${\tt PA}\subseteq T$. One reason is that in the monotheoretic version the theory $T$ (itself) should be $\Sigma_1$-complete (be able to prove all true arithmetical $\Sigma_1$-sentences) but in the bitheoretic version the $\Sigma_1$-completeness of a fixed theory (like ${\tt PA}$ or even its weak fragments) suffices. If that was sufficiently interesting, let us now have a look at (ii):  ${\sf Pr}_T(\ulcorner\varphi\rightarrow\psi\urcorner)\rightarrow\big[{\sf Pr}_T(\ulcorner\varphi\urcorner)\rightarrow{\sf Pr}_T(\ulcorner\psi\urcorner)\big]$ is true for any sentences $\varphi,\psi$ and any classical theory $T$. So, it must be provable in a sufficiently strong arithmetical theory (like ${\tt PA}$); whence we may have

${\tt PA}\vdash{\sf Pr}_T(\ulcorner\varphi\rightarrow\psi\urcorner)\rightarrow\big[{\sf Pr}_T(\ulcorner\varphi\urcorner)\rightarrow{\sf Pr}_T(\ulcorner\psi\urcorner)\big]$

\noindent for any theory $T$. Let us note that here we do not require (and do not need) the theory $T$ to contain ${\tt PA}$. Its strength over (ii) is more obvious.

Other than the above mathematical interests in the bitheoretic versions of the derivability conditions, Detlefsent \cite{detlefsen} sees two important philosophical reasons for the importance of the bitheoretic versions: ``The first is that it points up an element of unclarity in the usual `monotheoretic' formulations of'' G\"odel's Second Incompleteness Theorem.  ``In such formulations, some of the references to $T$ are references to it in its capacity {\em as representing theory} while others are references to it in its capacity {\em as represented theory}. The {\em justification} of the Derivability Conditions requires a clear demarcation of these roles. A justifiable constraint on the representing theory of a representational scheme can not generally be expected to be a justifiable constraint on the represented theory of that scheme, and {\em vice versa}. The justification of representational constraints therefore generally requires a distinction between the representing and represented theories of a representational scheme.'' (The emphasizes are Detlefsen's \cite{detlefsen}). ``The second reason the representing vs. represented theory distinction is important for our purposes is that ... certain applications ... require that we allow the two to be different. The particular application we have in mind is the application of [G\"odel's Second Incompleteness Theorem] to the evaluation of Hilbert's Program. It requires that we allow the representing theory to become as weak as (some codification of) finitary reasoning while, at the same time, allowing the represented theory to be as strong as the strongest classical theory that possesses the type of instrumental virtues for which Hilbert generally prized classical mathematics (e.g., various systems of set theory). If the [G\"odel's Second Incompleteness Theorem] phenomenon were to hold only for some environments containing finitary reasoning, and not for all of them, it would not be legitimate to take it as refuting Hilbert's Program because it would not then be an invariant feature of all (proper) representational environments. Justifications of the [Derivability Conditions] must therefor be valid not only in the monotheoretic settings but also in the appropriate bitheoretic settings.''

Let us now list the bitheoretic derivability conditions of \cite{detlefsen} for two theories $S$ (which is intended to be as weak as possible--representing finitary mathematics) and $T$ (which is intended to be as strong as possible--representing ideal mathematics):

\bigskip

\begin{tabular}{ll}
({\bf B}i) & $S\vdash{\sf Pr}_T(\ulcorner\varphi\urcorner)$ whenever $T\vdash\varphi$; \\
({\bf B}ii) & $S\vdash{\sf Pr}_T(\ulcorner\varphi\rightarrow\psi\urcorner)\rightarrow\big[{\sf Pr}_T(\ulcorner\varphi\urcorner)\rightarrow{\sf Pr}_T(\ulcorner\psi\urcorner)\big]$; \\
({\bf B}iii) & $S\vdash{\sf Pr}_T(\ulcorner\varphi\urcorner)\rightarrow {\sf Pr}_S(\ulcorner{\sf Pr}_T(\ulcorner\varphi\urcorner)\urcorner)$; \\
({\bf B}iv) & $S\vdash{\sf Pr}_S(\ulcorner\varphi\urcorner)\rightarrow {\sf Pr}_T(\ulcorner\varphi\urcorner)$; \\
({\bf B}v) & $S\vdash\mathscr{G} \longleftrightarrow \neg{\sf Pr}_T(\ulcorner \mathscr{G}\urcorner)$ for some sentence $\mathscr{G}$.
\end{tabular}

\bigskip

\noindent Then Detlefsen's Bi-G2 Lemma (\cite{detlefsen}, p.~48) proves that $S\vdash {\sf Con}(T)\longrightarrow\mathscr{G}$, which then implies G\"odel's Second Incompleteness Theorem by classical reasoning: for $S\subseteq T$ we have $S\not\vdash\mathscr{G}$  (by G\"odel's first incompleteness theorem) and so $S\not\vdash{\sf Con}(T)$; which  exactly negates Hilbert's Program: the consistency of ideal mathematics cannot be proved by finitary means.

\subsection{Arithmetical Theories: Minding P's and Q's}
Despite of the fact that usually G\"odel's first incompleteness theorem is proved for Peano's Arithmetic ${\tt PA}$, it holds for very weak fragments of ${\tt PA}$. It is interesting to note that by the techniques of G\"odel's theorem ${\tt PA}$ is proved to be non-finitely axiomatizable (see e.g. \cite{hajekpudlak}). But a magical theory, called Robinson's Arithmetic and denoted by ${\tt Q}$, was introduced in \cite{tarskibook} which has the following properties:

\begin{itemize}
\item ${\tt Q}$ is {finite}: ${\tt Q}={\tt PA}-\{\text{all induction axioms}\}+\forall x\exists y [x=0\vee x=S(y)]$;
\item ${\tt Q}$ is $\Sigma_1$-complete: $\Sigma_1\text{-}{\rm Th}(\mathbb{N})\subseteq{\tt Q}$;
\item ${\tt Q}$ is {\em essentially undecidable} (i.e., {\sc re}-incompletable):
 every {\sc re} and consistent extension of it is (undecidable and) incomplete.
\end{itemize}

\noindent The existence of a finitely axiomatized and undecidable theory immediately implies Church's (and Turing's) theorem on the undecidability of first order logic (giving a negative answer to the Entscheidungsproblem).
\begin{remark}{\rm
The somewhat mysterious symbol ${\tt Q}$ for this theory actually comes from its origin \cite{tarskibook} where (it was first introduced and) Peano's Arithmetic was denoted by ${\tt P}$ (nowadays shown by ${\tt PA}$), and the letter after P is of course Q. There is still another theory (with lots of interesting properties) called (again) Robinson's Arithmetic, denoted by ${\tt R}$, and its R (having nothing to do with Robinson)  just follows Q in the alphabet letters.
}\hfill\ding{71}\end{remark}
 G\"odel-Rosser's (stronger)  Incompleteness Theorem can be stated as ``no consistent and {\sc re} extension of ${\tt Q}$ is $\Pi_1$-decisive'', where a theory $T$ is called $\Gamma$-decisive, for a class $\Gamma$ of formulas, when for any $\phi\in\Gamma$ we have either $T\vdash\phi$ or $T\vdash\neg\phi$. In other words, the G\"odel-Rosser theorem states that for any theory $T$:

\smallskip

${\tt Q}\subseteq T \; \& \; T\in\Sigma_1 \; \& \; {\sf Con}(T) \Longrightarrow T\not\in\Pi_1\text{-Decisive}$.

\smallskip

\noindent In the next (final) subsection we will present a theory $\mathcal{Q}'$ such that for any theory $T$:

\smallskip

$\mathcal{Q}'\subseteq T \; \& \; T\in\Sigma_1 \; \& \; {\sf Con}(T) \Longrightarrow T\not\vdash{\sf Con}(T)$.

%\smallskip

\subsection{Second Thoughts on the Second Theorem}
Let us have another look at the derivability conditions from semantic point of view. As was mentioned before, proving them to hold in a particular theory could be difficult, but it is not too difficult to see right away that the followings hold:
\begin{itemize}\itemindent=1em
\item[(i')] if $T\vdash\varphi$, then $\mathcal{N}\models{\sf Pr}_T(\ulcorner\varphi\urcorner)$ (and so ${\tt Q}\vdash{\sf Pr}_T(\ulcorner\varphi\urcorner)$)
    for any {\sc re} theory $T$;
\item[(ii')] $\mathcal{N}\models{\sf Pr}_U(\ulcorner\varphi\rightarrow\psi\urcorner)\rightarrow\big[{\sf Pr}_U(\ulcorner\varphi\urcorner)\rightarrow{\sf Pr}_U(\ulcorner\psi\urcorner)\big]$, for any  {\sc re} theory $U$;
\item[(iii')] $\mathcal{N}\models\sigma\rightarrow{\sf Pr}_U(\ulcorner\sigma\urcorner)$, for any $\Sigma_1$-sentence $\sigma$ and any {\sc re} theory $U\supseteq{\tt Q}$.
\end{itemize}

\noindent Now, define

\bigskip

 $\mathcal{Q}'={\tt Q}\, \cup\{{\sf Pr}_U(\ulcorner\varphi\rightarrow\psi\urcorner)\rightarrow\big[{\sf Pr}_U(\ulcorner\varphi\urcorner)\rightarrow{\sf Pr}_U(\ulcorner\psi\urcorner)\big]\mid  U\textrm{ is an }\textsc{re}\textrm{ theory} \}$

\hspace{3em} $\cup\;\{\sigma\rightarrow{\sf Pr}_{U}(\ulcorner\sigma\urcorner)\mid \sigma \textrm{ is a }\Sigma_1\textrm{-sentence and }U\supseteq{\tt Q}\textrm{ is an }\textsc{re}\textrm{ theory}\}$.

\bigskip

\noindent By what was said above, it is clear that $\mathcal{N}\models\mathcal{Q}'$. And what was promised at the end of the last subsection can be proved rather easily:

\bigskip

\begin{theorem}[G\"odel's Second Incompleteness Theorem]\label{g2} For any consistent and {\sc re} extension $T$ of $\mathcal{Q}'$, we have $T\not\vdash{\sf Con}(T)$.
\end{theorem}

\bigskip

Let us postpone the proof for a moment, and pause more on $\mathcal{Q}'$. It would not be much of use if this theory were not {\sc re}. So, let us prove this very important fact before the main theorem:

%\newpage

\begin{theorem}
The theory $\mathcal{Q}'$ is {\sc re}.
\end{theorem}
\begin{proof}
Trivially, ${\tt Q}$, being a finite theory, is {\sc re}; and the class of all {\sc re} theories is {\sc re}, so is the class $\{{\sf Pr}_U(\ulcorner\varphi\rightarrow\psi\urcorner)\rightarrow\big[{\sf Pr}_U(\ulcorner\varphi\urcorner)\rightarrow{\sf Pr}_U(\ulcorner\psi\urcorner)\big]\mid  U\!\in\!\Sigma_1\}$. It remains to show that the class $\{\sigma\rightarrow{\sf Pr}_{U}(\ulcorner\sigma\urcorner)\mid \sigma\!\in\!\Sigma_1\text{-Sentences}\; \& \; {\tt Q}\!\subseteq\!U\!\in\!\Sigma_1\}$ is {\sc re} too. Again, the class of all $\Sigma_1$-sentences is {\sc re} and so is the class of $\Sigma_1$ theories; the finiteness of ${\tt Q}$ implies that the condition ${\tt Q}\!\subseteq\!U$ is equivalent to $U\vdash\bigwedge{\tt Q}$ (where $\bigwedge{\tt Q}$ denotes the conjunction of the finitely many axioms of ${\tt Q}$) which is an {\sc re} property (by a proof-search algorithm). Thus  $\sigma\!\in\!\Sigma_1\text{-Sentences}\; \& \; {\tt Q}\!\subseteq\!U\!\in\!\Sigma_1$ (for given sentence $\sigma$ and set of sentences  $U$) is an {\sc re} condition as well.
\end{proof}

So, we see that again the finiteness of the magical theory ${\tt Q}$ is essential for the recursive enumerability of $\mathcal{Q}'$; if ${\tt Q}$ were not finite,  then the condition ${\tt Q}\subseteq U$ would not be {\sc re} (for given {\sc re} theory $U$). And unfortunately, this is the best we can show for this theory. It would have been another magic to have $\mathcal{Q}'$ finitely axiomatized, or at least have a finitely axiomatized theory containing it. Indeed, there exists a finitely axiomatized theory that contains $\mathcal{Q}'$, and that is ${\tt I\Sigma_1}$, the fragment of ${\tt PA}$ whose induction axioms are restricted to $\Sigma_1$ formulas (see e.g. \cite{hajekpudlak}). But neither the finite axiomatizability of ${\tt I\Sigma_1}$ nor the fact that ${\tt I\Sigma_1}\vdash\mathcal{Q}'$ are easy to show (see the delicate  proofs in e.g. \cite{hajekpudlak}). So, the following easy proof could be difficult, if one wishes to show ${\bf T}\vdash\mathcal{Q}'$ for a particular theory ${\bf T}$.

\bigskip

\bigskip

\begin{proof}{\bf (of Theorem~\ref{g2})} By G\"odel-Carnap Diagonal Lemma there exists an arithmetical sentence $\mathscr{G}$   such that ${\tt Q}\vdash\mathscr{G}\longleftrightarrow[{\sf Pr}_T(\ulcorner\mathscr{G}\urcorner)\rightarrow\bot]$. The sentence $\mathscr{G}$ is equivalent to a $\Pi_1$-sentence and actually could be taken to be $\Pi_1$.  Now, $T\not\vdash\mathscr{G}$, since otherwise (if $T\vdash\mathscr{G}$ then) ${\tt Q}\vdash{\sf Pr}_T(\ulcorner\mathscr{G}\urcorner)$ and so ${\tt Q}\vdash\neg\mathscr{G}$ whence (by $T\supseteq\mathcal{Q}'\supseteq{\tt Q}$) $T\vdash\neg\mathscr{G}$, contradicting the consistency of $T$. Now, as $\neg\mathscr{G}\in\Sigma_1$ we have $\mathcal{Q}'\vdash\neg\mathscr{G}\rightarrow{\sf Pr}_T(\ulcorner\neg\mathscr{G}\urcorner)$ (noting that $T$ is an {\sc re} theory containing ${\tt Q}$). So,

$(\dag)\qquad \mathcal{Q}'\vdash\neg{\sf Pr}_T(\ulcorner\neg\mathscr{G}\urcorner)\rightarrow\mathscr{G}$.

\noindent On the other hand, by classical logic we have $\vdash \neg\mathscr{G}\rightarrow[\mathscr{G}\rightarrow\bot]$, which, by the definition of  $\mathcal{Q}'$, implies that
$\mathcal{Q}'\vdash{\sf Pr}_T(\ulcorner\neg\mathscr{G}\urcorner)\rightarrow\big[{\sf Pr}_T(\ulcorner\mathscr{G}\urcorner)\rightarrow{\sf Pr}_T(\ulcorner\bot\urcorner)\big]$, so

$(\ddag)\qquad \mathcal{Q}'\vdash\neg{\sf Pr}_T(\ulcorner\bot\urcorner) \longrightarrow \neg{\sf Pr}_T(\ulcorner\mathscr{G}\urcorner) \bigvee \neg{\sf Pr}_T(\ulcorner\neg\mathscr{G}\urcorner)$.

\noindent Now, by  $\mathscr{G}$'s property  we have $\mathcal{Q}'\vdash\neg{\sf Pr}_T(\ulcorner\mathscr{G}\urcorner)\rightarrow\mathscr{G}$ and by $(\dag)$ above $\mathcal{Q}'\vdash\neg{\sf Pr}_T(\ulcorner\neg\mathscr{G}\urcorner)\rightarrow\mathscr{G}$. Whence, $(\ddag)$ implies that
$\mathcal{Q}'\vdash\neg{\sf Pr}_T(\ulcorner\bot\urcorner) \longrightarrow \mathscr{G}$, or $\mathcal{Q}'\vdash{\sf Con}(T) \longrightarrow \mathscr{G}$. The desired conclusion $T\not\vdash{\sf Con}(T)$ follows from  $T\not\vdash\mathscr{G}$ proved above.
\end{proof}

\bigskip

\bigskip

\bigskip

\bigskip

%%%%%%%%%%%%%%%%%%%%%%%%%%%%%%%%%%%%
%%%%%%%%%%%%%%%%%%%%%%%%%
%%%%%%%%%%%%%%%%%%%%%%%%%%%%%%%%%%%%
%%%%%%%%%%%%%%%%%%%%%%%%%
%\newpage

\end{document}